\documentclass[a4paper,10pt]{article}
\usepackage[top=2.54cm,bottom=2.0cm,left=2.0cm,right=2.54cm, includeheadfoot]{geometry}

\usepackage[T1]{fontenc}
\usepackage[utf8]{inputenc}
\usepackage[english]{babel}
\usepackage{enumerate}
\usepackage{multirow,booktabs}
\usepackage[table]{xcolor}
\usepackage{fullpage}
\usepackage{lastpage}
\usepackage{indentfirst}

\usepackage{footnote}

\usepackage{amsmath,amsfonts,amssymb,amscd,amsthm}

\usepackage{bm}

\usepackage[all,2cell]{xy} \UseAllTwocells \SilentMatrices

\usepackage{hyperref}

\usepackage{subfiles}

\usepackage{multicol}

\usepackage{lineno}

\newtheorem{teo}{Theorem}[section]
\newtheorem{lem}[teo]{Lemma} 
\newtheorem{cor}[teo]{Corollary}
\newtheorem{prop}[teo]{Proposition}

\newtheorem{defn}[teo]{Definition}
\newtheorem{ex}[teo]{Example}

\newtheorem*{claim*}{Claim}

\newtheorem{rem}[teo]{Remark}

\newcommand{\cl}[1]{\left[#1\right]}
\newcommand{\nzd}{\operatorname{nzd}}

\newcommand{\U}{\mathcal U}

\begin{document}

\title{Marshall Quotients of the Rings $\mathbb Z/n\mathbb Z$}

\author{Lucas Colucci \\ Mathuzalem Ferreira de Lima\\ Kaique Matias de Andrade Roberto}

\date{}
\maketitle


\begin{abstract}
We study the Marshall quotient
\[
        M(n)=M(\mathbb Z/n\mathbb Z)
\]
obtained from the ring of integers modulo $n$ by quotienting by the square classes of non-zero-divisors.  Using elementary arithmetic of square classes modulo prime powers and the Chinese Remainder Theorem, we give an explicit description of these quotients and classify several of their structural properties.  We determine when the quotient relation is arithmetically elementary, when $M(n)$ is hyperbolic, when it can be real reduced, and when it can be formally real.  We also analyze the subset of invertible classes together with zero, proving exactly when it is a submultiring, when it is a hyperfield, and when it is hyperbolic.  The results provide a finite family of test examples for questions connecting multirings, special hyperfields, real semigroups, and abstract quadratic-form theory.
\end{abstract}



\section{Introduction}

There are several abstract theories of quadratic forms.  Early examples include abstract Witt rings, quaternionic structures, and Cordes schemes \cite{marshall1980abstract}.  Their common purpose was to isolate the formal properties of quadratic forms over fields and to use them in the construction, comparison, and obstruction theory of fields with prescribed quadratic-form behavior.  Later, Marshall's abstract spaces of orderings \cite{marshall1996spaces} provided an axiomatic setting for the reduced theory of quadratic forms and the theory of orderings.  The theory of special groups of Dickmann and Miraglia \cite{dickmann2000special} then gave a first-order framework treating reduced and non-reduced phenomena simultaneously.

The ring-theoretic case is more delicate.  Over a field, coefficients of quadratic forms are invertible once they are non-zero; over a general ring, zero-divisors and non-units cannot be ignored.  Real semigroups, abstract real spectra, and multirings were introduced in part to handle this difficulty.  In particular, Marshall's multiring approach \cite{marshall2006real} gives a language close to ordinary commutative algebra while retaining the multivalued addition naturally associated with square classes and quadratic forms.  This point of view is especially useful because multirings and hyperfields encode several structures appearing in abstract quadratic-form theory; see, for instance, \cite{ribeiro2016functorial,roberto2021quadratic,roberto2021ktheory}.

The purpose of this paper is to study, in a concrete finite situation, the Marshall quotients
\[
        M(A)=A/_m(A^2\cap\nzd(A))
\]
for the rings $A=\mathbb Z/n\mathbb Z$.  We write
\[
        M(n)=M(\mathbb Z/n\mathbb Z).
\]
Although the definition is simple, the resulting multiring remembers subtle arithmetic information about square classes modulo $n$, the zero-divisor structure of $\mathbb Z/n\mathbb Z$, and the behavior of sums of unit squares.

The guiding questions are the following.  For which $n$ is $M(n)$ essentially elementary?  For which $n$ is it hyperbolic?  Can $M(n)$ ever be real reduced or formally real?  If $M(n)^\times$ denotes the set of invertible elements of $M(n)$, when does
\[
        \U(n):=M(n)^\times\cup\{\cl0\}
\]
inherit a multiring or hyperfield structure from $M(n)$?  These questions are natural from the point of view of special hyperfields and real semigroups, where the invertible part of a structure often carries important quadratic-form information.

Our first result is an arithmetic model for $M(n)$.  If
\[
        S_n=(\mathbb Z_n^\times)^2
\]
is the subgroup of squares of units modulo $n$, then
\[
        \cl a=\cl b \quad\Longleftrightarrow\quad
        a\equiv bu^2\pmod n
        \quad\mbox{for some }u\in\mathbb Z_n^\times.
\]
The multivalued addition is given by the corresponding formula
\[
        \cl c\in \cl a+\cl b
\]
if and only if there are units $u,v,w$ modulo $n$ such that
\[
        cw^2\equiv au^2+bv^2\pmod n.
\]
This description reduces many questions to elementary congruences.

We then prove the following classification results.
\begin{enumerate}[i -]
    \item The quotient relation is trivial, i.e. $M(n)$ is the ordinary ring $\mathbb Z_n$, if and only if $n\mid24$.  If one also regards the prime quotients $M(p)$ as elementary square-class quotients, then the arithmetically elementary cases are exactly the integers $n$ such that $n\mid24$ or $n$ is prime.
    \item $M(n)$ is hyperbolic if and only if
    \[
        \gcd(n,30)=1.
    \]
    Equivalently, no prime divisor of $n$ is equal to $2$, $3$, or $5$.
    \item No $M(n)$ is real reduced.  Moreover, no $M(n)$ with $n\ge2$ is formally real.
    \item The set $\U(n)=M(n)^\times\cup\{\cl0\}$ is a submultiring of $M(n)$ if and only if $n=1$ or $n$ is prime.  Consequently, $\U(n)$ is a hyperfield if and only if $n$ is prime.
    \item The hyperfield $\U(n)$ is hyperbolic if and only if $n=p$ is a prime with $p\ge7$.  It is never real reduced.
\end{enumerate}

A small notational warning is useful.  For an odd prime $p$, the quotient $M(p)$ has three elements: the zero class, the square class, and the nonsquare class.  However, its hyperaddition depends on $p$ and it is not, in general, the sign hyperfield $Q_2$.  For this reason we avoid calling all such quotients ``trivial'' in the hyperfield-theoretic sense.  Instead, we use the term \emph{arithmetically elementary} for the finite square-class cases described above.

The paper is organized as follows.  Section \ref{preliminaries-section} recalls the definitions of multirings, hyperfields, Marshall quotients, hyperbolicity, real reducedness, and formal reality.  Section \ref{arithmetic-section} develops the arithmetic description of $M(n)$ and records the prime-power decompositions needed later.  Section \ref{small-section} gives the examples $M(9)$ and $M(21)$.  Section \ref{classification-section} contains the classification theorems.

\section{Multi Structures}\label{preliminaries-section}

We recall the multivalued structures used throughout the paper.  All multirings are assumed to be commutative.

\begin{defn}[Adapted from Definition 1.1 in \cite{marshall2006real}]\label{defn:multimonoid}
An \textbf{abelian or commutative multigroup} is a first-order structure $(G,\cdot,r,1)$, where $G$ is a non-empty set, $r:G\rightarrow G$ is a function, $1\in G$, and $\cdot\subseteq G\times G\times G$ is a ternary relation.  We write $c\in a\cdot b$ for $(a,b,c)\in\cdot$.  The following axioms are required for all $a,b,c,t\in G$:
\begin{description}
    \item[M1 -] If $c\in a\cdot b$, then $a\in c\cdot r(b)$ and $b\in r(a)\cdot c$.  We write $a\cdot b^{-1}$ for $a\cdot r(b)$.
    \item[M2 -] $b\in a\cdot1$ if and only if $a=b$.
    \item[M3 -] If there exists $x$ such that $x\in a\cdot b$ and $t\in x\cdot c$, then there exists $y$ such that $y\in b\cdot c$ and $t\in a\cdot y$.
    \item[M4 -] $c\in a\cdot b$ if and only if $c\in b\cdot a$.
\end{description}
The structure $(G,\cdot,1)$ is a \textbf{commutative multimonoid with unity} if it satisfies M3, M4, and $a\in1\cdot a$ for every $a\in G$.
\end{defn}

\begin{defn}[Adapted from Definition 2.1 in \cite{marshall2006real}]\label{defn:multiring}
A \textbf{multiring} is a sextuple $(R,+,\cdot,-,0,1)$, where $R$ is a non-empty set,
\[
        +:R\times R\rightarrow\mathcal P(R)\setminus\{\emptyset\}
\]
is a multivalued operation, $\cdot:R\times R\rightarrow R$ and $-:R\rightarrow R$ are functions, and $0,1\in R$, satisfying:
\begin{enumerate}[i -]
    \item $(R,+,-,0)$ is a commutative multigroup;
    \item $(R,\cdot,1)$ is a commutative monoid;
    \item $a\cdot0=0$ for every $a\in R$;
    \item if $c\in a+b$, then $c\cdot d\in a\cdot d+b\cdot d$.  Equivalently,
    \[
        (a+b)d\subseteq ad+bd.
    \]
\end{enumerate}
A multiring is a \textbf{hyperring} if the distributive inclusion above is always an equality.  A multiring is a \textbf{multifield} if every non-zero element is invertible.  In this situation the notions of multifield and hyperfield coincide.
\end{defn}

\begin{ex}\label{ex:basic-examples}
$ $
\begin{enumerate}[a -]
    \item Every ordinary commutative ring is a multiring by replacing each usual sum $a+b$ with the singleton $\{a+b\}$.
    \item The Krasner hyperfield $K=\{0,1\}$ has the usual product and multivalued sum given by $x+0=0+x=\{x\}$ and $1+1=\{0,1\}$.
    \item The sign hyperfield $Q_2=\{-1,0,1\}$ has the usual product and multivalued sum
    \[
    \begin{cases}
        0+x=x+0=\{x\}, & x\in Q_2,\\
        1+1=\{1\},\quad (-1)+(-1)=\{-1\},\\
        1+(-1)=(-1)+1=\{-1,0,1\}.
    \end{cases}
    \]
\end{enumerate}
\end{ex}

\begin{defn}[Marshall quotient, Example 2.6 in \cite{marshall2006real}]\label{defn:strangeloc}
Let $A$ be a multiring and let $S\subseteq A$ be a multiplicative subset with $1\in S$.  Define an equivalence relation on $A$ by
\[
        a\sim b \quad\Longleftrightarrow\quad as=bt
        \quad\mbox{for some }s,t\in S.
\]
We denote the equivalence class of $a$ by $\overline a$ and set
\[
        A/_mS=\{\overline a:a\in A\}.
\]
The operations are
\[
\begin{array}{rcl}
\overline a\,\overline b&=&\overline{ab},\\[2mm]
-\overline a&=&\overline{-a},\\[2mm]
\overline a+\overline b&=&\{\overline c:cv=as+bt\mbox{ for some }s,t,v\in S\}.
\end{array}
\]
\end{defn}

\begin{rem}
Marshall quotients are central in the multiring treatment of quadratic-form structures.  They appear naturally in the comparison of multirings with special groups and real semigroups; see \cite{ribeiro2016functorial,roberto2021quadratic,roberto2021ktheory}.
\end{rem}

\begin{defn}[Hyperbolic multiring]\label{defn:hyperbolic}
A multiring $R$ is \textbf{hyperbolic} if
\[
        1-1=R.
\]
Equivalently, every element of $R$ is represented by a difference of two copies of $1$ in the multivalued addition.
\end{defn}

\begin{defn}[Real reduced hyperfield]\label{defn:mfrealreduced}
A hyperfield $F$ is \textbf{real reduced} if $a^3=a$ for every $a\in F$ and
\[
        a\in1+1\quad\Longrightarrow\quad a=1.
\]
\end{defn}

\begin{defn}[Real reduced multiring]\label{defn:mrrealreduced}
A multiring $A$ is \textbf{real reduced} if the following conditions hold for all $a,b,c,d\in A$:
\begin{enumerate}[i -]
    \item $1\ne0$;
    \item $a^3=a$;
    \item $c\in a+ab^2$ implies $c=a$;
    \item if $c,d\in a^2+b^2$, then $c=d$.
\end{enumerate}
\end{defn}

\begin{defn}\label{defn:formally-real}
A multiring $A$ is \textbf{formally real} if
\[
        -1\notin \sum A^2,
\]
that is, if $-1$ is not contained in any finite multivalued sum of squares of elements of $A$.
\end{defn}

\section{The Arithmetic Model for $M(n)$}\label{arithmetic-section}

For $n\ge1$, write
\[
        \mathbb Z_n=\mathbb Z/n\mathbb Z.
\]
When no confusion is possible, we identify a residue class with its representative in $\{0,1,\ldots,n-1\}$.  The non-zero-divisors of $\mathbb Z_n$ are precisely the units $\mathbb Z_n^\times$.  Hence, for $A=\mathbb Z_n$, the multiplicative set in Definition \ref{defn:strangeloc} is
\[
        S_n=(\mathbb Z_n^\times)^2=\{u^2:u\in\mathbb Z_n^\times\}.
\]

\begin{defn}\label{defn:Mn}
For $n\ge1$, define
\[
        M(n):=M(\mathbb Z_n)=\mathbb Z_n/_m S_n.
\]
The class of $a\in\mathbb Z_n$ in $M(n)$ will be denoted by $\cl a$.
\end{defn}

\begin{lem}\label{lem:basic-description}
Let $a,b,c\in\mathbb Z_n$.  Then:
\begin{enumerate}[i -]
    \item $\cl a=\cl b$ in $M(n)$ if and only if
    \[
        a\equiv bu^2\pmod n
    \]
    for some $u\in\mathbb Z_n^\times$;
    \item $\cl c\in\cl a+\cl b$ in $M(n)$ if and only if there are units $u,v,w\in\mathbb Z_n^\times$ such that
    \[
        cw^2\equiv au^2+bv^2\pmod n.
    \]
\end{enumerate}
\end{lem}

\begin{proof}
By Definition \ref{defn:strangeloc}, $\cl a=\cl b$ if and only if $as=bt$ for some $s,t\in S_n$.  Since $S_n$ is a subgroup of $\mathbb Z_n^\times$, this is equivalent to $a=bts^{-1}$ with $ts^{-1}\in S_n$, i.e. to $a\equiv bu^2\pmod n$ for some unit $u$.

The formula for addition is the same translation of Definition \ref{defn:strangeloc}.  Namely, $\cl c\in\cl a+\cl b$ if and only if $cr=as+bt$ for some $r,s,t\in S_n$.  Writing $r=w^2$, $s=u^2$, and $t=v^2$ gives the desired congruence.
\end{proof}

\begin{cor}\label{cor:gcd-invariant}
If $\cl a=\cl b$ in $M(n)$, then
\[
        \gcd(a,n)=\gcd(b,n).
\]
In particular, multiplying by square classes of units preserves the prime-power divisibility pattern of a residue modulo $n$.
\end{cor}

\begin{proof}
If $a\equiv bu^2\pmod n$ and $u$ is a unit modulo $n$, then multiplication by $u^2$ is an automorphism of the additive group of $\mathbb Z_n$.  It preserves the greatest common divisor with $n$.
\end{proof}

\begin{prop}[Chinese remainder decomposition]\label{prop:crt-Mn}
Let
\[
        n=\prod_{i=1}^r p_i^{\alpha_i}
\]
be the prime factorization of $n$.  Then the Chinese Remainder Theorem induces an isomorphism of multirings
\[
        M(n)\cong \prod_{i=1}^r M(p_i^{\alpha_i}).
\]
\end{prop}

\begin{proof}
The Chinese Remainder Theorem gives a ring isomorphism
\[
        \mathbb Z_n\cong\prod_{i=1}^r\mathbb Z_{p_i^{\alpha_i}}.
\]
Under this isomorphism, units correspond to tuples of units, and squares of units correspond to tuples of squares of units.  Therefore $S_n$ corresponds to $\prod_i S_{p_i^{\alpha_i}}$.  The defining equivalence relation and the multivalued addition in Lemma \ref{lem:basic-description} are coordinatewise under this identification.  Hence the induced map on Marshall quotients is an isomorphism of multirings.
\end{proof}

\begin{prop}[Odd prime powers]\label{prop:odd-prime-powers}
Let $p$ be an odd prime, let $\alpha\ge1$, and let $s$ be any unit whose image modulo $p$ is a quadratic non-residue.  Then
\[
        M(p^\alpha)=\{\cl0,\cl1,\cl s,\cl p,\cl{ps},\ldots,
        \cl{p^{\alpha-1}},\cl{p^{\alpha-1}s}\}.
\]
In particular, $|M(p^\alpha)|=2\alpha+1$.
\end{prop}

\begin{proof}
Every non-zero residue modulo $p^\alpha$ can be written uniquely in the form $p^j u$, where $0\le j\le \alpha-1$ and $u$ is a unit modulo $p^\alpha$.  For odd $p$, the quotient
\[
        \mathbb Z_{p^\alpha}^\times/(\mathbb Z_{p^\alpha}^\times)^2
\]
has two elements, represented by $1$ and by any lift $s$ of a quadratic non-residue modulo $p$.  Thus $u$ is equivalent, up to multiplication by a square of a unit, either to $1$ or to $s$.  It follows that the class of $p^j u$ is either $\cl{p^j}$ or $\cl{p^j s}$.

The classes listed are distinct because multiplication by a unit square preserves the $p$-adic valuation, and because $1$ and $s$ represent the two distinct unit square classes.  Adding the zero class gives exactly $2\alpha+1$ classes.
\end{proof}

\begin{lem}\label{lem:square-subgroup-trivial}
The following are equivalent:
\begin{enumerate}[i -]
    \item $M(n)$ is the ordinary ring $\mathbb Z_n$, i.e. the quotient relation is equality;
    \item $S_n=\{1\}$;
    \item every unit $u\in\mathbb Z_n^\times$ satisfies $u^2=1$;
    \item $n\mid24$.
\end{enumerate}
\end{lem}

\begin{proof}
The equivalence between (i), (ii), and (iii) follows immediately from Lemma \ref{lem:basic-description}.  We prove the equivalence with (iv).

Write
\[
        n=2^\alpha m,
\]
where $m$ is odd.  By the Chinese Remainder Theorem, the condition that every unit modulo $n$ has square $1$ is equivalent to the same condition modulo each prime-power factor of $n$.  For an odd prime power $p^e$, the group $\mathbb Z_{p^e}^\times$ is cyclic of order $p^{e-1}(p-1)$.  All its elements have square $1$ if and only if this order is at most $2$, which happens only for $p^e=3$.

For powers of $2$, one checks directly that all odd residues square to $1$ modulo $2$, $4$, and $8$, while this fails modulo $16$ since $3^2\equiv9\not\equiv1\pmod{16}$.  Therefore the possible powers of $2$ are $1,2,4,8$, and the only possible odd prime-power factor is $3$.  Hence
\[
        n\in\{1,2,3,4,6,8,12,24\},
\]
which is equivalent to $n\mid24$.
\end{proof}

\section{Some Calculations for Small $n$}\label{small-section}

We record two examples that will be used as references later.

\begin{ex}[$M(9)$]\label{ex:M9}
The squares of units modulo $9$ are
\[
        S_9=(\mathbb Z_9^\times)^2=\{1,4,7\}.
\]
Therefore
\[
        M(9)=\{\cl0,\cl1,\cl2,\cl3,\cl6\},
\]
where
\[
\begin{array}{rcl}
\cl0&=&\{0\},\\
\cl1&=&\{1,4,7\},\\
\cl2&=&\{2,5,8\},\\
\cl3&=&\{3\},\\
\cl6&=&\{6\}.
\end{array}
\]
The invertible classes are
\[
        M(9)^\times=\{\cl1,\cl2\},
\]
and hence
\[
        \U(9)=M(9)^\times\cup\{\cl0\}=\{\cl0,\cl1,\cl2\}.
\]
However, $\U(9)$ is not closed under the induced multivalued addition: for example,
\[
        \cl3\in\cl1+\cl2,
\]
because $3\equiv1+2\pmod9$.
\end{ex}

\begin{ex}[$M(21)$]\label{ex:M21}
The squares of units modulo $21$ are
\[
        S_{21}=\{1,4,16\}.
\]
Thus
\[
        M(21)=\{\cl0,\cl1,\cl2,\cl3,\cl5,\cl7,\cl9,\cl{10},\cl{14}\},
\]
where
\[
\begin{array}{rclcrcl}
\cl0&=&\{0\},       && \cl1&=&\{1,4,16\},\\
\cl2&=&\{2,8,11\},  && \cl3&=&\{3,6,12\},\\
\cl5&=&\{5,17,20\}, && \cl7&=&\{7\},\\
\cl9&=&\{9,15,18\}, && \cl{10}&=&\{10,13,19\},\\
\cl{14}&=&\{14\}.&&&
\end{array}
\]
The invertible classes are
\[
        M(21)^\times=\{\cl1,\cl2,\cl5,\cl{10}\}.
\]
Moreover,
\[
        \cl3\in\cl1+\cl2,
\]
since $3\equiv1+2\pmod{21}$.  The class $\cl3$ is not invertible, so $\U(21)$ is not a submultiring of $M(21)$.
\end{ex}

\section{Classification Results}\label{classification-section}

\subsection{Arithmetically elementary quotients}

The word ``trivial'' can be misleading in this context.  For example, if $p$ is an odd prime, then $M(p)$ has three elements, but its hyperaddition need not be the hyperaddition of the sign hyperfield $Q_2$.  We therefore use the following terminology.

\begin{defn}\label{defn:arith-elementary}
We say that $M(n)$ is \textbf{arithmetically elementary} if either
\begin{enumerate}[i -]
    \item the quotient relation is equality, so that $M(n)=\mathbb Z_n$ as an ordinary multiring; or
    \item $n$ is prime, so that $M(n)$ is the square-class quotient of a finite field.
\end{enumerate}
\end{defn}

\begin{prop}\label{prop:prime-case}
Let $p$ be an odd prime.  Then
\[
        M(p)=\{\cl0,\cl1,\cl s\},
\]
where $s$ is any quadratic non-residue modulo $p$.
\end{prop}

\begin{proof}
This is Proposition \ref{prop:odd-prime-powers} with $\alpha=1$.
\end{proof}

\begin{teo}\label{teo:arith-elementary}
The quotient $M(n)$ is arithmetically elementary if and only if
\[
        n\mid24
        \quad\mbox{or}\quad
        n\mbox{ is prime}.
\]
\end{teo}

\begin{proof}
By Lemma \ref{lem:square-subgroup-trivial}, the quotient relation is equality if and only if $n\mid24$.  By definition, every prime $n=p$ gives an arithmetically elementary square-class quotient.  Conversely, these are exactly the two alternatives in Definition \ref{defn:arith-elementary}.
\end{proof}

\begin{rem}
If one uses the word ``trivial'' only for the equality quotient $M(n)=\mathbb Z_n$, then the classification is simply $n\mid24$.  If one also regards the prime square-class quotients as elementary finite building blocks, then Theorem \ref{teo:arith-elementary} gives the corresponding classification.
\end{rem}

\subsection{Hyperbolic quotients}

We now classify the integers $n$ for which $M(n)$ is hyperbolic.

\begin{lem}\label{lem:prime-hyp}
Let $p$ be a prime.  Then $M(p)$ is hyperbolic if and only if $p\ge7$.
\end{lem}

\begin{proof}
For $p=2$, the only unit square is $1$, and $1-1=0$.  Hence $\cl1\notin\cl1-\cl1$.  The same obstruction occurs for $p=3$, since the only unit square modulo $3$ is $1$.  For $p=5$, the unit squares are $1$ and $4$, and the differences of two unit squares are congruent only to $0$, $2$, or $3$ modulo $5$.  Thus no element of the square class $\{1,4\}$ belongs to $\cl1-\cl1$.

Now suppose $p\ge7$.  It is enough to prove that every non-zero residue $x\in\mathbb F_p^\times$ is a difference of two non-zero squares.  Choose $r\in\mathbb F_p^\times$ such that
\[
        r^2\ne x\quad\mbox{and}\quad r^2\ne -x.
\]
This is possible because the two equations $r^2=x$ and $r^2=-x$ have altogether at most four solutions, while $|\mathbb F_p^\times|=p-1\ge6$.  Define
\[
        a=\frac{1}{2}\left(r+xr^{-1}\right),
        \qquad
        b=\frac{1}{2}\left(r-xr^{-1}\right).
\]
The choice of $r$ guarantees that $a$ and $b$ are non-zero.  Moreover,
\[
        a^2-b^2=(a+b)(a-b)=r\cdot xr^{-1}=x.
\]
Thus every non-zero class, and also the zero class, belongs to $\cl1-\cl1$.  Hence $M(p)$ is hyperbolic.
\end{proof}

\begin{lem}\label{lem:prime-power-hyp}
Let $p\ge7$ be a prime and let $\alpha\ge1$.  Then $M(p^\alpha)$ is hyperbolic.
\end{lem}

\begin{proof}
Let $x\in\mathbb Z_{p^\alpha}$.  We show that $x$ is congruent modulo $p^\alpha$ to a difference of two unit squares.

If $x$ is a unit, choose a unit $r$ modulo $p^\alpha$ such that
\[
        r^2\not\equiv x\pmod p,
        \qquad
        r^2\not\equiv -x\pmod p.
\]
This is possible by the same counting argument as in Lemma \ref{lem:prime-hyp}.  Put
\[
        a=2^{-1}(r+xr^{-1}),
        \qquad
        b=2^{-1}(r-xr^{-1})
\]
in $\mathbb Z_{p^\alpha}$.  Then $a$ and $b$ are units, and
\[
        a^2-b^2\equiv x\pmod{p^\alpha}.
\]

If $x$ is not a unit and $x\ne0$, write $x=p^j u$ with $1\le j\le\alpha-1$ and $u$ a unit.  Set
\[
        a=2^{-1}(u+p^j),
        \qquad
        b=2^{-1}(u-p^j).
\]
Both $a$ and $b$ are units, because they are congruent to $2^{-1}u$ and $2^{-1}u$ up to sign modulo $p$.  Moreover,
\[
        a^2-b^2=(a+b)(a-b)=u\,p^j=x.
\]
Finally, $0=1^2-1^2$.  Therefore every class of $M(p^\alpha)$ belongs to $\cl1-\cl1$.
\end{proof}

\begin{teo}\label{teo:hyperbolic}
For $n\ge1$, the multiring $M(n)$ is hyperbolic if and only if
\[
        \gcd(n,30)=1.
\]
Equivalently, $2\nmid n$, $3\nmid n$, and $5\nmid n$.
\end{teo}

\begin{proof}
Suppose first that $\gcd(n,30)=1$.  Then every prime divisor $p$ of $n$ satisfies $p\ge7$.  By Lemma \ref{lem:prime-power-hyp}, every prime-power factor $M(p^\alpha)$ of $M(n)$ is hyperbolic.  The Chinese Remainder decomposition in Proposition \ref{prop:crt-Mn} then implies that $M(n)$ is hyperbolic.

Conversely, suppose that $2$, $3$, or $5$ divides $n$.  We show that $\cl1\notin\cl1-\cl1$.

If $2\mid n$, then every unit square is congruent to $1$ modulo $2$, and every difference of two unit squares is congruent to $0$ modulo $2$.  Hence such a difference cannot represent the class $\cl1$.

If $3\mid n$, the same argument works modulo $3$, because every unit square is congruent to $1$ modulo $3$.

If $5\mid n$, the unit squares modulo $5$ are $1$ and $4$.  Their differences are congruent only to $0$, $2$, or $3$ modulo $5$, never to a square modulo $5$.  Hence the class $\cl1$ cannot belong to $\cl1-\cl1$ in $M(n)$.

Thus $M(n)$ is not hyperbolic when $n$ is divisible by $2$, $3$, or $5$.
\end{proof}

\subsection{Real reducedness and formal reality}

The next results show that the quotients $M(n)$ do not produce real reduced or formally real examples.

\begin{lem}\label{lem:two-squares-local}
For every prime $p$, there exist non-zero elements $u,v\in\mathbb F_p$ such that
\[
        u^2+v^2
\]
is not a non-zero square in $\mathbb F_p$.
\end{lem}

\begin{proof}
For $p=2$, take $u=v=1$; then $u^2+v^2=0$.

Now let $p$ be odd.  If $-1$ is a square modulo $p$, take $u=1$ and choose $v$ with $v^2=-1$.  Then $u^2+v^2=0$.

Assume, therefore, that $-1$ is not a square modulo $p$.  Let $Q$ be the set of non-zero squares in $\mathbb F_p$.  If $1+x\in Q$ for every $x\in Q$, then $1+Q\subseteq Q$.  Since translation by $1$ is injective and $-1\notin Q$, the set $1+Q$ has the same cardinality as $Q$ and does not contain $0$.  Hence $1+Q=Q$.  Iterating, one obtains $m+Q=Q$ for every integer $m$, which would force $Q=\mathbb F_p$, a contradiction.  Thus there is $x\in Q$ such that $1+x$ is not a non-zero square.  Writing $x=v^2$ and taking $u=1$ proves the lemma.
\end{proof}

\begin{teo}\label{teo:not-real-reduced}
There is no positive integer $n$ such that $M(n)$ is a real reduced multiring.
\end{teo}

\begin{proof}
For $n=1$, one has $1=0$, so the first axiom in Definition \ref{defn:mrrealreduced} fails.  Let $n\ge2$.

Suppose first that $p^2\mid n$ for some prime $p$.  Then the class $\cl p$ is non-zero in $M(n)$.  However,
\[
        (\cl p)^3=\cl{p^3}.
\]
By Corollary \ref{cor:gcd-invariant}, equivalence in $M(n)$ preserves the greatest common divisor with $n$.  Since $\gcd(p,n)$ and $\gcd(p^3,n)$ have different $p$-adic valuations, $\cl{p^3}\ne\cl p$.  Hence the identity $a^3=a$ fails.

It remains to treat the case where $n$ is square-free.  If $2\mid n$, then
\[
        \cl2\in\cl1+\cl1
\]
but $\cl2\ne\cl1$, again by Corollary \ref{cor:gcd-invariant}.  Thus the axiom $c\in a+ab^2\Rightarrow c=a$ fails with $a=b=1$.

Assume now that $n$ is odd and square-free.  Choose a prime divisor $p$ of $n$.  By Lemma \ref{lem:two-squares-local}, there exist non-zero residues $u,v$ modulo $p$ such that $u^2+v^2$ is not a non-zero square modulo $p$.  By the Chinese Remainder Theorem, choose units $a,b$ modulo $n$ whose images modulo $p$ are $u,v$, respectively.  Then
\[
        \cl{a^2+b^2}\in\cl1+\cl1.
\]
Modulo $p$, the element $a^2+b^2$ is not a non-zero square.  Therefore $\cl{a^2+b^2}\ne\cl1$ in $M(n)$.  Again the axiom $c\in1+1\Rightarrow c=1$ fails.  Hence $M(n)$ is not real reduced.
\end{proof}

\begin{teo}\label{teo:not-formally-real}
For every $n\ge2$, the multiring $M(n)$ is not formally real.
\end{teo}

\begin{proof}
By Lagrange's four-square theorem, the integer $n-1$ is a sum of four squares:
\[
        n-1=x_1^2+x_2^2+x_3^2+x_4^2
\]
for suitable integers $x_1,x_2,x_3,x_4$.  Reducing modulo $n$, we obtain
\[
        -1\equiv x_1^2+x_2^2+x_3^2+x_4^2\pmod n.
\]
Therefore
\[
        \cl{-1}\in \cl{x_1}^2+\cl{x_2}^2+\cl{x_3}^2+\cl{x_4}^2
\]
in $M(n)$.  Hence $-1$ is a sum of squares in $M(n)$, so $M(n)$ is not formally real.
\end{proof}

\subsection{The invertible part}

We now study the subset of invertible classes.  Let
\[
        \U(n)=M(n)^\times\cup\{\cl0\}.
\]

\begin{lem}\label{lem:units-Mn}
For $a\in\mathbb Z_n$, the class $\cl a$ is invertible in $M(n)$ if and only if $a$ is a unit modulo $n$.
\end{lem}

\begin{proof}
If $a$ is a unit modulo $n$, then $\cl a\,\cl{a^{-1}}=\cl1$, so $\cl a$ is invertible.  Conversely, if $\cl a\,\cl b=\cl1$, then $\cl{ab}=\cl1$.  By Corollary \ref{cor:gcd-invariant}, $\gcd(ab,n)=1$.  Hence $a$ is a unit modulo $n$.
\end{proof}

\begin{teo}\label{teo:unit-submultiring}
The set $\U(n)=M(n)^\times\cup\{\cl0\}$ is a submultiring of $M(n)$ if and only if $n=1$ or $n$ is prime.
\end{teo}

\begin{proof}
If $n=1$, the assertion is immediate.  If $n=p$ is prime, every non-zero class in $M(p)$ is invertible, so $\U(p)=M(p)$ and $\U(p)$ is a submultiring.

Conversely, suppose that $n$ is composite.  We prove that $\U(n)$ is not closed under multivalued addition.

If $n$ is even, then $n>2$ and
\[
        \cl2\in\cl1+\cl1.
\]
The class $\cl2$ is neither zero nor invertible in $M(n)$, so $\cl2\notin\U(n)$.

Now suppose that $n$ is odd and composite.  Choose a prime-power divisor $p^\alpha\Vert n$.  By the Chinese Remainder Theorem, choose a unit $b$ modulo $n$ such that
\[
        b\equiv p-1\pmod{p^\alpha}
\]
and
\[
        b\equiv1\pmod{q^\beta}
\]
for every other prime-power divisor $q^\beta\Vert n$.  Then $b$ is a unit modulo $n$, so $\cl b\in\U(n)$.  Put $c=1+b$.  We have
\[
        \cl c\in\cl1+\cl b.
\]
Moreover, $c$ is divisible by $p$, so it is not a unit modulo $n$.  Since $n$ is composite, the congruences above also show that $c$ is not congruent to zero modulo $n$.  Thus $\cl c\ne\cl0$ and $\cl c\notin M(n)^\times$, so $\cl c\notin\U(n)$.  Hence $\U(n)$ is not a submultiring.
\end{proof}

\begin{cor}\label{cor:unit-hyperfield}
The set $\U(n)$ is a hyperfield with the operations induced from $M(n)$ if and only if $n$ is prime.
\end{cor}

\begin{proof}
If $n$ is prime, then $\U(n)=M(n)$ and $M(n)$ is the Marshall quotient of a field by a subgroup of its multiplicative group; hence it is a hyperfield.  Conversely, if $\U(n)$ is a hyperfield with the induced operations, then it is in particular a submultiring of $M(n)$.  By Theorem \ref{teo:unit-submultiring}, $n=1$ or $n$ is prime.  The case $n=1$ is excluded for hyperfields because $0=1$.
\end{proof}

\begin{cor}\label{cor:unit-hyp}
The hyperfield $\U(n)$ is hyperbolic if and only if $n=p$ is a prime with $p\ge7$.
\end{cor}

\begin{proof}
By Corollary \ref{cor:unit-hyperfield}, $\U(n)$ is a hyperfield only when $n$ is prime.  In that case $\U(p)=M(p)$, and the result follows from Lemma \ref{lem:prime-hyp}.
\end{proof}

\begin{cor}\label{cor:unit-not-real-reduced}
There is no positive integer $n$ such that $\U(n)$ is a real reduced hyperfield with the operations induced from $M(n)$.
\end{cor}

\begin{proof}
If $\U(n)$ is a hyperfield, then $n=p$ is prime by Corollary \ref{cor:unit-hyperfield}.  Thus $\U(p)=M(p)$.  The proof of Theorem \ref{teo:not-real-reduced}, applied in the prime case, shows that $\cl1+\cl1$ contains an element different from $\cl1$.  Therefore $M(p)$ is not real reduced as a hyperfield.
\end{proof}

\section{Conclusion}

We have rewritten the study of the Marshall quotients $M(n)$ in arithmetic terms.  The basic equivalence relation is multiplication by square classes of units, and the multivalued addition is controlled by congruences of the form
\[
        cw^2\equiv au^2+bv^2\pmod n.
\]
This allows the structure of $M(n)$ to be analyzed prime-power by prime-power.

The classification obtained here is as follows.  The equality quotient $M(n)=\mathbb Z_n$ occurs exactly for $n\mid24$, while the arithmetically elementary quotients are obtained precisely when $n\mid24$ or $n$ is prime.  Hyperbolicity is governed by the small primes $2$, $3$, and $5$: one has $M(n)$ hyperbolic if and only if $\gcd(n,30)=1$.  On the other hand, no $M(n)$ is real reduced, and no $M(n)$ with $n\ge2$ is formally real.  Finally, the invertible classes together with zero form a submultiring only for $n=1$ or $n$ prime; in the prime case this gives a hyperfield, hyperbolic exactly for primes $p\ge7$, and never real reduced.

These results show that the finite quotients $M(n)$ provide a useful testing ground for questions about multirings, hyperfields, and the passage from ring-theoretic quadratic-form data to invertible square-class structures.

\section*{Declarations}
\textbf{Conflict of interest.} The authors declare no conflict of interest.

\section*{Acknowledgments:}
The first author was supported by FAPESB (EDITAL FAPESB No. 012/2022 -- UNIVERSAL -- No. APP0044/2023). The third author was supported by the S\~ao Paulo Research Foundation (FAPESP, Brazil), thematic project {\em Rationality, logic and probability -- RatioLog}, grant  2020/16353-3 and by a post-doctoral grant from FAPESP, grant 2024/18577-7.

\end{document}